\documentclass[11pt]{amsart} 

\usepackage[utf8]{inputenc} 

\usepackage{lipsum}

\usepackage{geometry} 
\usepackage{graphicx} 

\usepackage[english]{babel}

\usepackage{booktabs} 
\usepackage{array} 
\usepackage{paralist} 
\usepackage{verbatim} 
\usepackage{subfig} 

\usepackage{caption}

\usepackage{amsfonts,amsmath,amssymb,bbm,amsthm,amsbsy,amscd}
\usepackage{mathtools}
\usepackage{pinlabel}

\usepackage{color}
\usepackage{pdfcolmk}

\usepackage{fancyhdr} 
\usepackage{tikz}
\usepackage{tikz-cd}

\usepackage{IEEEtrantools}

\newenvironment{equ*}[1]{\begin{IEEEeqnarray*}{#1}}{\end{IEEEeqnarray*}}

\usepackage[pdfborder={0 0 0 [0 0 ]},bookmarksdepth=3,bookmarksopen=true]{hyperref}
\usepackage[capitalize]{cleveref}
\usepackage[ps,all,arc,rotate]{xy}
\usepackage{bbm} 

\hypersetup{
    colorlinks=true,
    linkcolor=black,
    citecolor=black,
    filecolor=black,
    urlcolor=black,
}

\usepackage{cleveref}

\newenvironment{frcseries}{\fontfamily{frc}\selectfont}{}

\makeatletter                                                   
\newtheorem*{rep@theorem}{\rep@title}
\newcommand{\newreptheorem}[2]{%
\newenvironment{rep#1}[1]{%
 \def\rep@title{#2 \ref{##1}}%
 \begin{rep@theorem}}%
 {\end{rep@theorem}}}
\makeatother

\newtheorem{thm}{Theorem}[section]
\newtheorem{lemma}[thm]{Lemma}

\crefname{lemma}{Lemma}{Lemmata}
\newtheorem{prop}[thm]{Proposition}
\newtheorem{corr}[thm]{Corollary}
\newtheorem{claim}[thm]{Claim}

\newtheorem*{thm*}{Theorem}
\newtheorem*{lemma*}{Lemma}
\newtheorem*{prop*}{Proposition}
\newtheorem*{corr*}{Corollary}
\newtheorem*{claim*}{Claim}
\newtheorem*{convergence*}{$(\star)$ Condition}
\newtheorem*{pure_theorem}{Purity Theorem~\ref{theorem:pure}}
\newtheorem*{inert_theorem}{Inertia Theorem~\ref{theorem:inert}}

\theoremstyle{remark}
\newtheorem{rmk}[thm]{Remark}

\newtheorem*{rmk*}{Remark}
\newtheorem*{conj*}{Conjecture}
\newtheorem*{quest*}{Question}

\theoremstyle{definition}
\newtheorem{defn}[thm]{Definition}
\newtheorem{exmp}[thm]{Example}

\newtheorem*{defn*}{Definition}
\newtheorem*{exmp*}{Example}

\newreptheorem{theorem}{Theorem}
\newreptheorem{corollary}{Corollary}
\newreptheorem{proposition}{Proposition}

\newcommand{\R}{\mathbb{R}}

\newcommand{\Z}{\mathbb{Z}}
\newcommand{\N}{\mathbb{N}}

\newcommand{\Shat}{\widehat{S}}
\newcommand{\Sigmahat}{\widehat{\Sigma}}
\newcommand{\intrm}{\mathrm{int}}

\newcommand{\Mod}{\mathrm{Mod}}
\newcommand{\PMod}{\mathrm{PMod}}

\newcommand{\PSL}{\mathrm{PSL}}

\newcommand{\Aut}{\mathrm{Aut}}
\newcommand{\Ray}{\mathcal{R}}

\newcommand{\normseq}{{\lbrace N_n\rbrace}}
\newcommand{\Ndiv}{N_{\text{div}}}

\newcommand{\defeq}{\vcentcolon=}

\title{Normal Subgroups of Big Mapping Class Groups}

\author{Danny Calegari}
\address{Department of Mathematics\\ University of Chicago\\ Chicago, Illinois, USA}
\email[D.~Calegari]{dannyc@math.uchicago.edu}

\author{Lvzhou Chen}
\address{Department of Mathematics\\ University of Texas at Austin\\ Austin, Texas, USA}
\email[L.~Chen]{lvzhou.chen@math.utexas.edu}

\begin{document}

\begin{abstract}
Let $S$ be a surface and let $\Mod(S,K)$ be the mapping class group of $S$ 
permuting a Cantor subset $K \subset S$. We prove two structure theorems
for normal subgroups of $\Mod(S,K)$.

(Purity:) if $S$ has finite type, every normal subgroup of $\Mod(S,K)$ either 
contains the kernel of the forgetful map to the mapping class group of $S$, or it is `pure' 
--- i.e.\/ it fixes the Cantor set pointwise. 

(Inertia:) for any $n$ element subset $Q$ of the Cantor set, there is a forgetful map 
from the pure subgroup $\PMod(S,K)$ of $\Mod(S,K)$ to the mapping class group of 
$(S,Q)$ fixing $Q$ pointwise.
If $N$ is a normal subgroup of $\Mod(S,K)$ contained in $\PMod(S,K)$, its image $N_Q$ is likewise
normal. We characterize exactly which finite-type normal subgroups $N_Q$ arise this way.

Several applications and numerous examples are also given.
\end{abstract}

\maketitle

\setcounter{tocdepth}{2}
\tableofcontents

\section{Introduction}

In recent years there has been an surge of interest in the theory of so-called
{\em big mapping class groups}, i.e.\/ mapping class groups of surfaces of infinite type.
There are many motivations for studying such objects (see e.g. \cite{Aramayona_Vlamis} for
an excellent recent survey) but our motivation comes from (2 real dimensional or 1 complex
dimensional) dynamics. Thus we are especially interested in big mapping class groups
of a certain kind, which we now explain.

If $S$ is a surface of finite type, a hyperbolic dynamical system $\Gamma$ acting on $S$ will often
determine a dynamically distinguished compact subset $\Lambda$. The terminology for
$\Lambda$ varies depending on the context: it is an {\em attractor} if $\Gamma$ is an Iterated Function System;
a {\em limit set} if $\Gamma$ is a Kleinian group; a {\em Julia set} if $\Gamma$ is the set
of iterates of a rational map; or a {\em hyperbolic set} if $\Gamma$ is an Axiom A diffeomorphism.
Often, $\Lambda$ will be a Cantor set, and the topology of the moduli spaces of pairs $(\Gamma,\Lambda)$
is related to the mapping class group of $S$ minus a Cantor set and its
subgroups. Because there is no natural choice of coordinates on $\Lambda$, 
it is the {\em normal} subgroups
of such mapping class groups that will arise, and therefore this paper is devoted to
exploring the structure and classification of such subgroups.

\subsection{Statement of results}

Let $S$ be a surface of finite type (i.e. $S$ has finite genus and finitely many punctures). Let $\Mod(S)$ denote
the mapping class group of $S$ and for any compact totally disconnected subset 
$Q$ of $S$ let $\Mod(S,Q)$ denote the mapping class group of $S$ rel. $Q$ (i.e.\/ the group of
homotopy classes of orientation-preserving self-homeomorphisms of $S$ permuting $Q$ as a set)
and $\PMod(S,Q)$ the {\em pure} mapping class group of $S$ rel. $Q$ (i.e.\/ the group of
homotopy classes of orientation-preserving self-homeomorphisms of $S$ fixing $Q$ pointwise).
Pure subgroups of big mapping class groups and their properties are studied 
e.g.\/ in \cite{Aramayona_Patel_Vlamis} and \cite{Patel_Vlamis}. 
Note that $\PMod(S,Q)$ is a normal subgroup of $\Mod(S,Q)$. When $Q$ is finite, we write
the cardinality $n:=|Q|$. For finite $Q$ the groups $\Mod(S,Q)$ and $\PMod(S,Q)$ 
depend only on the cardinality of $Q$, up to conjugation by elements of $\Mod(S,Q)$. 
By abuse of notation, we denote these equivalence classes of groups by $\Mod(S,n)$ 
and $\PMod(S,n)$ respectively.

There are forgetful maps $\Mod(S,Q) \to \Mod(S)$ for all $Q$, and 
$\PMod(S,R) \to \PMod(S,Q)$ for all pairs ($Q,R$) with $Q \subset R$.
When $K \subset S$ is a Cantor set and $N \subset \PMod(S,K)$ is a normal subgroup of
$\Mod(S,K)$, the image $N_n \subset \PMod(S,n)$ is a well-defined normal subgroup of $\Mod(S,n)$.

Our first main theorem is the following:

\begin{pure_theorem}
Let $S$ be a connected orientable surface of finite type, and let $K$
be a Cantor set in $S$. Then any normal subgroup of $\Mod(S,K)$ either contains the
kernel of the forgetful map to $\Mod(S)$, or it is contained in $\PMod(S,K)$.
\end{pure_theorem}

The two cases in the theorem correspond to normal subgroups of $\Mod(S,K)$ of countable and uncountable index respectively. 
It implies that all normal subgroups of countable index are pulled back from normal subgroups of $\Mod(S)$.

When $S$ is the plane or the 2-sphere, $\Mod(S)$ is trivial. It follows that every
proper normal subgroup of $\Mod(S,K)$ is contained in $\PMod(S,K)$; in particular,
its index in $\Mod(S,K)$ is {\em uncountable}. Thus the mapping class group
of the plane or the sphere minus a Cantor set admits no nontrivial homomorphism to a
countable group. This fact was proved independently by Nicholas Vlamis \cite{Vlamis}.

We also obtain the abelianization and generating sets of $\Mod(S,K)$; See Section~\ref{sec: app}.

\medskip

Our second main theorem concerns the subgroups $N_n \subset \PMod(S,n)$ that arise as
the image of subgroups $N \subset \PMod(S,K)$ normal in $\Mod(S,K)$. To state this 
theorem we must give the definition of an {\em inert} subgroup of $\PMod(S,n)$.

Let $S$ be any surface and let $Q \subset S$
be a finite subset with cardinality $n$. Let $D_Q \subset S$ be a collection of $n$
disjoint closed disks, each centered at some point of $Q$. Let $\PMod(S,D_Q)$ be the mapping
class group of $S$ fixing $D_Q$ pointwise (by both homeomorphisms and homotopies). 
There is a forgetful map $\PMod(S,D_Q) \to \PMod(S,Q)$
and as is well-known, this is a $\Z^n$ central extension (see e.g.\/ \cite{Farb_Margalit} Prop~3.19).

Let $\alpha \in \PMod(S,Q)$ be any element, and let $\hat{\alpha}$ be some lift to
$\PMod(S,D_Q)$. Let $f:Q \to Q$ be any map (possibly not injective). The {\em insertion}
$f^*\hat{\alpha} \in \PMod(S,n)$ is the ($\Mod(S,n)$)-conjugacy class of element defined as
follows. Let $\pi:D_Q \to Q$ be the map that takes each component of $D_Q$ to its center, and 
let $\tilde{f}:Q \to D_Q$ be any injective map for which the composition $\pi\tilde{f}=f$.
Then $f^*\hat{\alpha}$ is the image of $\hat{\alpha}$ under the forgetful map to 
$\PMod(S,\tilde{f}(Q))$ which may be
canonically identified (up to conjugacy in $\Mod(S,n)$) with $\PMod(S,n)$. Note that the
class of $f^*\hat{\alpha}$ is invariant under pre-composition of $f$ with a permutation of $Q$,
and it depends only on the cardinalities of the preimages (under $f$) of the elements of $Q$.

A subgroup $N \subset \PMod(S,n)$ normal in $\Mod(S,n)$ is said to be {\em inert} if,
after fixing some identification of $\PMod(S,n)$ with $\PMod(S,Q)$, for every
$\alpha \in N$ there is a lift $\hat{\alpha}$ in $\PMod(S,D_Q)$ so that for every
$f:Q \to Q$ the insertion $f^*\hat{\alpha}$ is in $N$. For more on this definition see 
\S~\ref{subsection:inertia}.

Our second main theorem is the following:

\begin{inert_theorem}
Let $S$ be any connected, orientable surface, and let $K$ be a Cantor set in $S$. 
A subgroup $N_n$ of $\PMod(S,n)$ is equal to the image of some $\Mod(S,K)$-normal
pure subgroup $N \subset \PMod(S,K)$ under the forgetful map if and only if it is inert.
\end{inert_theorem}

\subsection*{Acknowledgment}
The authors thank Ian Biringer for his questions about the normal closure of torsion elements in big mapping class groups, which inspired this work. The authors also thank Lei Chen, Tom Church, Benson Farb, Elizabeth Field, Hannah Hoganson, Rylee Lyman, Justin Malestein, Dan Margalit, Jeremy Miller, Andy Putman, Jing Tao, and Nicholas Vlamis for helpful discussions.

\section{The Purity Theorem}

The purpose of this section is to prove the following theorem:

\begin{thm}[Purity Theorem]\label{theorem:pure}
Let $S$ be a connected orientable surface of finite type, and let $K$
be a Cantor set in $S$. Then any normal subgroup of $\Mod(S,K)$ either contains the
kernel of the forgetful map to $\Mod(S)$, or it is contained in $\PMod(S,K)$.
\end{thm}

An immediate corollary is:

\begin{corr}\label{corr: maximal normal}
Any homomorphism from $\Mod(S,K)$ to a countable group factors through $\Mod(S,K)\to\Mod(S)$.
Thus if $\Mod(S)$ is trivial (for instance if $S$ is the plane or the 2-sphere) then
$\Mod(S,K)$ has no nontrivial homomorphisms to a countable group.
\end{corr}

If $\Mod(S)$ is trivial then $\Mod(S,K)/\PMod(S,K)$ is isomorphic to $\Aut(K)$, the
group of self-homeomorphisms of a Cantor set. So Corollary~\ref{corr: maximal normal} gives
a new proof that $\Aut(K)$ is simple, which is a theorem of Anderson \cite{Anderson}.

\subsection{Applications}\label{sec: app}

In this section we give some applications of the Purity Theorem~\ref{theorem:pure}.

The first application is to give a computation of $H_1(\Mod(S,K))$, the abelianization of $\Mod(S,K)$.

\begin{thm}\label{thm: H_1}
Let $S$ be a surface of finite type, and let $K$
be a Cantor set in $S$. Then the forgetful map $\Mod(S,K) \to \Mod(S)$ induces an isomorphism
on $H_1$; i.e.\/
$$H_1(\Mod(S,K)) \cong H_1(\Mod(S)).$$
\end{thm}
\begin{proof}
Let $N$ be the commutator subgroup of $\Mod(S,K)$. Its image in $\Aut(K)$ is nontrivial (actually,
its image is all of $\Aut(K)$ by Anderson's theorem, but we do not use this fact). Thus $N$
is not contained in $\PMod(S,K)$, so by Theorem~\ref{theorem:pure} it contains the kernel of
the forgetful map $\pi:\Mod(S,K) \to \Mod(S)$.

Since $\pi$ is surjective, $\pi(N)$ is equal to the commutator subgroup of $\Mod(S)$. Thus
$$H_1(\Mod(S,K))=\Mod(S,K)/N\cong \Mod(S)/\pi(N)= H_1(\Mod(S)).$$
\end{proof}

The next application lets us determine the normal closure of certain subsets of elements.
If $T$ is a subset of a group $G$, the {\em normal closure} of $T$ is the subgroup of $G$
(algebraically) generated by conjugates of $T$. A subset $T$ is said to {\em normally generate} $G$ if its
normal closure is $G$.

\begin{lemma}\label{lemma: normal gen criterion}
Let $T \subset \Mod(S,K)$ be a subset that is not contained in $\PMod(S,K)$. If the
image $\pi(T)$ of $T$ in $\Mod(S)$ normally generates $\Mod(S)$ then $T$ normally 
generates $\Mod(S,K)$.
\end{lemma}
\begin{proof}
Let $N$ denote the normal closure of $T$. Then $N$ is not contained in $\PMod(S,K)$ by
hypothesis, so by Theorem~\ref{theorem:pure}, it contains the kernel of $\pi$. It follows
that $N$ is equal to $\pi^{-1}$ of the normal closure of $\pi(T)$.
\end{proof}

One interesting case is to take $T$ to be the set of torsion elements.

\begin{lemma}\label{lemma: PMap torsion-free}
	$\PMod(S,K)$ is torsion-free.
\end{lemma}
\begin{proof}
We may identify $\Mod(S,K)$ with $\Mod(S-K)$ in the obvious way. Every torsion element
$g$ of $\Mod(S-K)$ is realized by a periodic homeomorphism $f$ of $S-K$ by \cite[Thm.~2]{Afton_Calegari_Chen_Lyman}
which extends to a periodic homeomorphism of $S$ permuting $K$.

But any nontrivial periodic (orientation-preserving) homeomorphism $f$ of a finite type surface 
fixes only finitely many points. This follows from a theorem of K\'er\'ekjart\`o
\cite{Kerekjarto}, which says that any finite order orientation-preserving homeomorphism
of the plane is conjugate to a rotation (and therefore fixes at most two points). This
is well-known to experts, but we give an argument to reduce to K\'er\'ekjart\`o's theorem.
First extend the homeomorphism $f$ over the punctures
(if any) to reduce to the case of a homeomorphism of a closed surface which we also denote $S$. 
If there are infinitely many fixed points, there is a fixed point $p$ which is 
a limit of fixed points $p_i$. If $\gamma_i$ is a short path from $p$ to $p_i$ then evidently
$f(\gamma_i)$ is homotopic to $\gamma_i$ rel. endpoints for all sufficiently big $i$. Let
$\tilde{f}$ be a lift of $f$ to the universal cover $\tilde{S}$ of $S$ fixing some point
$\tilde{p}$ that covers $p$. Then $\tilde{f}$ has a finite power that covers the identity
on $S$, and since it fixes $\tilde{p}$ it {\em is} the identity; i.e.\/ $\tilde{f}$ is a torsion
element acting on $\tilde{S}$. Furthermore, $\tilde{f}$ fixes lifts $\tilde{p}_i$ of $p_i$ joined
to $\tilde{p}$ by lifts $\tilde{\gamma}_i$ of $\gamma_i$. But this contradicts K\'er\'ekjart\`o's
theorem.
\end{proof}

\begin{thm}\label{thm: normal gen by tor}
Let $S$ be a surface of finite type, and let $K$ be a 
Cantor set in $S$. Then $\Mod(S,K)$ is generated by torsion unless $S$ has genus two and $5k+4$ punctures for some $k\ge0$. 
Moreover, a torsion element $g$
in $\Mod(S,K)$ normally generates $\Mod(S,K)$ if and only if its image in $\Mod(S)$ normally
generates $\Mod(S)$.
\end{thm}
\begin{rmk}
By \cite[Thm.~1.1 and Prop.~3.3]{Lanier_Margalit}, 
when $S$ is closed and has genus at least $3$, a torsion element of $\Mod(S)$ normally generates 
$\Mod(S)$ if and only if it is not a hyperelliptic involution. 
In this sense, most torsion elements normally generate the entire mapping class group.
\end{rmk}

\begin{proof}[Proof of Theorem \ref{thm: normal gen by tor}]
The second assertion directly follows from Lemmas \ref{lemma: normal gen criterion} and \ref{lemma: PMap torsion-free}.
	
For the first assertion, let $T$ be the set of all torsion elements in $\Mod(S,K)$. 
In view of Lemmas \ref{lemma: normal gen criterion} and \ref{lemma: PMap torsion-free}, 
it suffices to check that under the forgetful map, the image of $T$
generates $\Mod(S)$. We claim that $\pi(T)$ is exactly the set of torsion elements 
in $\Mod(S)$. For, each torsion element $g$ in $\Mod(S)$ is represented (by Nielsen realization)
by a periodic homeomorphism $f$ of $S$. If $K' \subset S$ is any Cantor set, so is the union $K$ of its
orbits under the (finite) set of powers of $f$. Thus $f$ permutes some Cantor set $K$ in $S$
and represents a finite order element of $\Mod(S,K)$ mapping to $g$.

For any finite type surface $S$, the group $\Mod(S)$ is generated by torsion elements unless $S$ has genus two and $5k+4$ punctures for some $k\ge0$; see \cite{Luo} and \cite{Korkmaz}. This completes the proof of the first assertion.
\end{proof}

Similar results have been obtained independently by Justin Malestein and Jing Tao for surfaces of infinite genus when the space of ends has certain self-similar structure \cite{Malestein_Tao}.

\subsection{Proof of the Purity Theorem when $S$ has at most one puncture}\label{subsec: Purity one puncture}

We shall prove the Purity Theorem in the case where $S$ has exactly one puncture, denoted $\infty$.
The case of closed $S$ may be deduced from this case by the following Lemma \ref{lemma: closed case}. 
Let's fix notation $\Sigma = S - K$, and define $\Shat\defeq S\cup\{\infty\}$ and 
$\Sigmahat\defeq\Sigma\cup\{\infty\}=\Shat - K$. There is a canonical isomorphism
of $\Mod(S,K)$ with $\Mod(\Sigma)$; throughout the remainder of this section we usually 
discuss $\Mod(\Sigma)$.

\begin{lemma}\label{lemma: closed case}
	Suppose every normal subgroup $N$ of $\Mod(\Sigma)$ either lies in $\PMod(\Sigma)$ or 
	contains the kernel of $\Mod(\Sigma) \to \Mod(S)$.
	Then every normal subgroup $\widehat{N}$ of $\Mod(\Sigmahat)$ either lies in $\PMod(\Sigmahat)$ 
	or contains the kernel of $\Mod(\Sigmahat) \to \Mod(\Shat)$.
\end{lemma}
\begin{proof}
Let's denote the kernel of the forgetful map $\pi_S:\Mod(\Sigma) \to \Mod(S)$ by
$\Mod(\Sigma)_0$, and likewise the kernel of the forgetful map $\pi_{\Shat}:\Mod(\Sigmahat) \to \Mod(\Shat)$ by
$\Mod(\Sigmahat)_0$.

We have the following commutative diagram of short exact sequences, where each row is the
Birman exact sequence obtained by point-pushing $\infty$:
	$$
	\begin{CD}
		1 @>>> 	\pi_1(\Sigmahat)	@>>>	\Mod(\Sigma)		@>p>>	\Mod(\Sigmahat)	@>>>	1\\
		@. 		@VVV						@V\pi_{S}VV				@V\pi_{\Shat}VV				@.\\
		1 @>>>  \pi_1(\Shat)		@>>>	\Mod(S)	@>>>	\Mod(\Shat) @>>>	1,
	\end{CD}
	$$
and	where the vertical maps $\pi_S$ and $\pi_{\Shat}$ are the forgetful homomorphisms.
	
	We first show that $p(\Mod(\Sigma)_0)=\Mod(\Sigmahat)_0$. For any $\hat{g}\in \Mod(\Sigmahat)_0$, 
	let $g\in \Mod(\Sigma)$ be any element with $p(g)=\hat{g}$. 
	Then $\pi_{S}(g)$ must come from point-pushing infinity around a loop $\gamma\in\pi_1(\Shat)$. 
	We can lift the isotopy class of $\gamma$ to some $\tilde{\gamma}\in \pi_1(\Sigmahat)$. 
	Then as a mapping class in $\Mod(\Sigma)$, $\tilde{\gamma}$ 
	and $g$ have the same image under $\pi_S$. Hence $\tilde{\gamma}^{-1}\cdot g\in \Mod(\Sigma)_0$ and $p(\tilde{\gamma}^{-1}\cdot g)=\hat{g}$, Thus
	$p(\Mod(\Sigma)_0)=\Mod(\Sigmahat)_0$ as claimed.
	
	Now for any normal subgroup $\widehat{N}$ of $\Mod(\Sigmahat)$ not contained in $\PMod(\Sigmahat)$, let $N\defeq p^{-1}(\widehat{N})$. Then $N$ is not contained in $\PMod(\Sigma)$, and thus by our assumption it contains $\Mod(\Sigma)_0$. Therefore, by what we showed above, we have
	$\Mod(\Sigmahat)_0=p(\Mod(\Sigma)_0)\subset p(N)=\widehat{N}$.
\end{proof}

The remainder of this section is devoted to the proof of Theorem~\ref{theorem:pure} under
the hypothesis that $S$ is a finite genus surface with exactly one puncture. The case when $S$ has multiple punctures can be proved similarly, which we explain in Section \ref{subsec: Purity general case}.

\begin{defn}\label{def: dividing disk}
	We say a disk $D$ in $S$ is a \emph{dividing disk} if 
	both the interior and exterior of $D$ intersect $K$ while its boundary does not.
	Note that $\infty$ must lie in the exterior of a dividing disk $D$. 
\end{defn}

We say a mapping class $g$ is supported in a dividing disk $D$ if it can be realized 
as a homeomorphism that is the identity outside $D$. Denote by $\Ndiv$ the subgroup of $\Mod(\Sigma)$ 
generated by all mapping classes supported in dividing disks. Then $\Ndiv$ is normal in $\Mod(\Sigma)$.
For brevity of notation, throughout the next section we write 
$\Gamma$ for $\Mod(\Sigma)$ and $\Gamma^0$ $\Mod(\Sigma)_0$ (i.e.\/ the kernel of $\Mod(\Sigma) \to \Mod(S)$).

The proof of Theorem~\ref{theorem:pure} has two steps. First (Lemma~\ref{lemma: normal act trivially}) 
we will show the normal closure of any $g \in \Gamma - \PMod(\Sigma)$ contains $\Ndiv$. 
Second (Proposition~\ref{prop: normal closure large}) we will show that $\Ndiv = \Gamma^0$. This
will complete the proof. 

\begin{lemma}\label{lemma: normal act trivially}
	The normal closure of any $g\in \Gamma - \PMod(\Sigma)$ contains $\Ndiv$.
\end{lemma}
\begin{proof}
	Let $x\in K$ be a point such that $x\neq g(x)$. Then there is a small closed dividing disk $D$
	containing $x$ such that $D \cap g(D)=\emptyset$. Since any embedding of a Cantor set in a disk
	is standard, there exists a sequence of dividing disks $\{D_n\}_{n\ge 0}$ inside $D$ converging to $x$
	and a homeomorphism $h$ supported in $D$ preserving $K\cap D$ and fixing $x$ such that $h(D_n)=D_{n+1}$ for all $n\ge 0$;
	See Figure \ref{fig: shrinking_disks}.
	
	\begin{figure}[hbtp]
		\labellist
		\small 
		
		\pinlabel $x$ at 205 45
		\pinlabel $D$ at 217 30
		\pinlabel $D_0$ at 108 20
		\pinlabel $h$ at 124 82
		\pinlabel $D_1$ at 151 32
		\pinlabel $h$ at 164 75
		\pinlabel $D_2$ at 188 40
		\endlabellist
		\centering
		\includegraphics[scale=1]{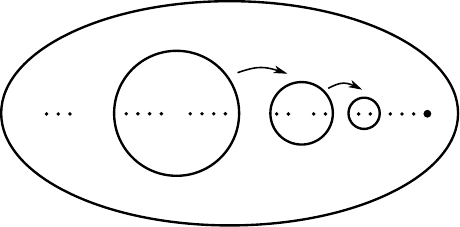}
		\caption{A sequence of dividing disks $D_n$ in $D$ converging to $x$ and 
			a mapping class $h$ supported in $D$ with $h(D_n)=D_{n+1}$}
		\label{fig: shrinking_disks}
	\end{figure}
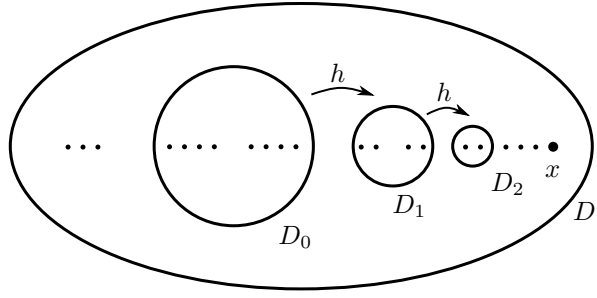
	
	Denote $g(D)$ by $E$ and $g(D_n)$ by $E_n$. Let $b\in \Gamma$ be a mapping class 
	such that
	$b|_{D}=g|_{D}$ and $b|_{E}=h\circ (g|_{D})^{-1}=h\circ (g^{-1})|_{E}$. Such a mapping class exists since any given 
	orientation-preserving homeomorphism of the boundary of $\Sigma - \intrm(D\sqcup E)$ 
	extends to an orientation-preserving homeomorphism on all of $\Sigma$.
	
	Now let $f$ be any mapping class in $\Gamma$ supported in a dividing disk. Up to a conjugation, 
	we may assume $f$ is supported in $D_0$.
	Then the infinite product $x_f\defeq \prod_{n\ge 0} h^nfh^{-n}$ is a well defined mapping class 
	in $\Gamma$. 
	
	Furthermore $a_f\defeq[x_f,g]=(x_f g x_f^{-1})g^{-1}$ lies in $\langle\langle g\rangle\rangle$, 
	and therefore so does $a_f ba_f b^{-1}$. We claim $a_f ba_f b^{-1}=f$. 
	Indeed, $a_f$ and $ba_f b^{-1}$ are 
	supported in $(\sqcup_{n\ge 0}D_n)\sqcup(\sqcup_{n\ge 0}E_n)\sqcup\{x,g(x)\}$ and 
	$(\sqcup_{n\ge 1}D_n)\sqcup(\sqcup_{n\ge 0}E_n)\sqcup\{x,g(x)\}$ respectively. 
	Note that for all $n\ge0$, we have
	$$(ba_f b^{-1})|_{D_{n+1}}=hg^{-1}\cdot (a_f)|_{E_n}\cdot gh^{-1}=
	h\cdot (x_f^{-1})|_{D_n}\cdot h^{-1}=(a_f^{-1})|_{D_{n+1}},$$
	and
	$$(ba_f b^{-1})|_{E_n}=g\cdot (a_f)|_{D_n}\cdot g^{-1}=
	g\cdot (x_f)|_{D_n}\cdot g^{-1}=(a_f^{-1})|_{E_n}.$$
	Thus $a_f$ and $ba_f b^{-1}$ cancel out on all $D_{n+1}$ and $E_n$ for $n\ge 0$. Hence 
	$a_f ba_f b^{-1}$ is supported in $D_0$, on which it agrees with $a_f$ and thus with $f$. 
	Thus $a_f ba_f b^{-1}=f$.
\end{proof}

The next step is to show that $\Ndiv=\Gamma^0$. This will be accomplished in a series of lemmas. 
First note that $\Ndiv\subset \Gamma^0$ since each mapping class
supported in a disk becomes trivial under the forgetful map to $\Mod(S)$. 
To prove the converse, we shall use notation and terminology 
consistent with \cite{Calegari_Chen_rigidity}. Recall that a {\em short ray} in $S$ is an
isotopy class of proper embedded ray from $\infty$ to some point in $K$, and a {\em lasso} is a
homotopically essential properly embedded copy of $\R$ in $S$ from $\infty$ to $\infty$.
For any short ray $r$, let $\Gamma_r$ be the stabilizer of $r$ and 
let $\Gamma_{(r)}$ be the subgroup of mapping classes that are the identity in a neighborhood of $r$. 

When $S=\R^2$, Lemmas 6.6 and 6.8 from \cite{Calegari_Chen_rigidity} show that the 
normal closure of $\Gamma_{(r)}$ is equal to $\Gamma$. This proves $\Ndiv=\Gamma=\Gamma^0$ 
in the special case of $S=\R^2$ since each mapping class in $\Gamma_{(r)}$ 
is supported in a dividing disk in this situation. 

For general $S$, for any (finite) collection $L$ of disjoint lassos on $\Sigma$ 
that are disjoint from $r$ and cut $S$ into a single disk, 
let $\Gamma_{L,(r)}$ denote the subgroup of $\Gamma_{(r)}$ consisting of mapping classes 
that fix the isotopy class of each lasso in $L$. 
We say that such a collection $L$ is \emph{filling} (with respect to $r$).

\begin{lemma}\label{lemma: GammaL(r) contained in N}
	For each filling collection $L$, the group $\Gamma_{L,(r)}$ is a subgroup of $\Ndiv$.
\end{lemma}
\begin{proof}
	By definition, the lassos in $L$ cut $S$ into a disk $D_1$ 
	which contains the ray $r$. Each mapping class in $\Gamma_{L,(r)}$ 
	is represented by a homeomorphism $h$ that is supported in $D_1$ and 
	is the identity in an open disk neighborhood $D_2\subset D_1$ of $r$. 
	Thus $h$ is supported in $D_1-D_2$, which is a dividing disk 
	if we choose $D_2$ suitably; see Figure \ref{fig: dividingsupp}. 
	Hence each mapping class of $\Gamma_{L,(r)}$ lies in $\Ndiv$.
\end{proof}

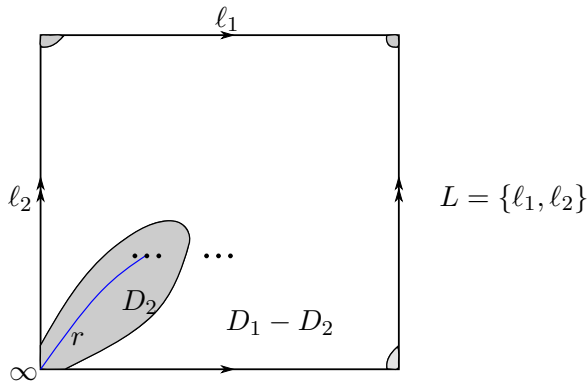
\begin{figure}[hbtp]
	\labellist
	\pinlabel $\infty$ at -8 0
	\pinlabel $D_2$ at 65 45
	\pinlabel $r$ at 26 22
	\pinlabel $D_1-D_2$ at 153 32
	\pinlabel $\ell_2$ at -10 110
	\pinlabel $\ell_1$ at 120 223
	\pinlabel $L=\{\ell_1,\ell_2\}$ at 300 110
	\endlabellist
	\centering
	\includegraphics[scale=0.6]{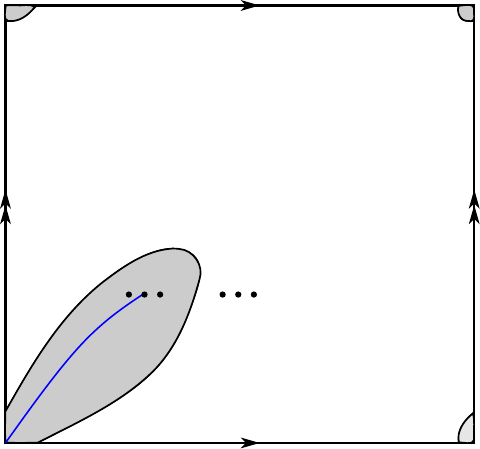}
	\caption{The dividing disk $D_1 - D_2$ for a suitable choice of $D_2$ in the case of $S=T^2-\infty$}
	\label{fig: dividingsupp}
\end{figure}

Our strategy to prove $\Gamma^0\subset \Ndiv$ is to show that these subgroups $\Gamma_{L,(r)}$ of $\Ndiv$ together generate $\Gamma^0$ as we vary $L$ and $r$. This is accomplished in two steps: First we show that as we vary $L$ fixing $r$, these subgroups generate $\Gamma_{(r)}\cap \Gamma^0$ in Lemma \ref{lemma: vary L}; Second we show that as we vary $r$, the subgroups $\Gamma_{(r)}\cap \Gamma^0$ generate $\Gamma^0$ using Lemmas \ref{lemma: (r) to r} and \ref{lemma: generate Gamma^0}.

\begin{lemma}\label{lemma: vary L}
	For any simple ray $r$, the group generated by all $\Gamma_{L,(r)}$ over all collections of
	lassos $L$ filling with respect to $r$, is equal to $\Gamma_{(r)} \cap \Gamma^0$.
\end{lemma}
\begin{proof}
	Every element of $\Gamma_{L,(r)}$ is supported in a disk and thus is trivial 
	as a mapping class on the surface $S$. Hence each $\Gamma_{L,(r)}$ is a subgroup of $\Gamma_{(r)}\cap \Gamma^0$.
	
	Conversely, it suffices to factorize an arbitrary mapping class $g$ of 
	$\Gamma_{(r)}\cap \Gamma^0$ as the product of finitely many mapping classes in $\Gamma_{(r)}$,
	each of which fixes a filling collection of lassos with respect to $r$. 
	
	Pick an arbitrary filling collection $L_0=\{\ell_1,\cdots,\ell_n\}$ with respect to $r$. Then $g\ell_1$ is isotopic to $\ell_1$ on $S$ minus a neighborhood of $r$ since $g\in \Gamma_{(r)}\cap \Gamma^0$. 
	
	\begin{claim}\label{claim: factorization}
		There is a factorization $g=h_1g_1$, where $g_1\in\Gamma_{(r)}\cap \Gamma^0$ fixes $\ell_1$ 
		and $h_1$ is the product of finitely many mapping classes in $\Gamma_{(r)}$ each of
		which fixes a filling collection of lassos with respect to $r$. 
	\end{claim}
	\begin{proof}
		First consider the case where $g\ell_1$ and $\ell_1$ are disjoint. 
		Then they cobound some open bigon $B$ in $S - r$ since they are isotopic. 
		We may assume that $K\cap B$ is nonempty and thus is also a Cantor set, 
		since otherwise $g\ell_1$ and $\ell_1$ are isotopic in $\Sigma$ and $g$ itself 
		can be represented by a homeomorphism fixing $\ell_1$. 
		Also note that $K\cap B$ is a proper subset of $K$ because the endpoint of $r$ (which lies in $K$)
		is outside $B$. 
		We can choose disjoint lassos $\gamma_a,\gamma_b,\gamma_c$ in $B$ that cut both $B$ and $K\cap B$ into four parts, where $\gamma_b$ is in the middle, $\gamma_a$ is closer to $g\ell_1$ and $\gamma_c$ is closer to $\ell_1$; see the left of Figure \ref{fig: cutbigon}. 
		We can choose lassos $\ell'_2,\cdots,\ell'_n$ on $\Sigma$ such that each $\ell'_i$ is disjoint from $\ell_1,g\ell_1$ and homotopic to $\ell_i$ in $S - r$.
		Then $L=\{\ell_1,\ell'_2,\cdots,\ell'_n\}$ is a filling collection, which cuts $\Sigma$ into a disk $D$ minus a Cantor set.
		There is a homeomorphism $h_1$ preserving $K$ and fixing each lasso in $L$ (in particular $\ell_1$) such that 
		\begin{itemize}
			\item $h_1$ is the identity in a neighborhood of $r$; and
			\item $h_1g\ell_1=\gamma_b$ and $h_1\gamma_b=\gamma_c$.
		\end{itemize}
		Refer to Figure \ref{fig: cutbigon} to see the effect of $h_1$ in an example, where the disks $D_1$ and $D_2$ are used to track how the Cantor subset outside $B$ changes under $h_1$.
		Then $h_1\in \Gamma_{L,(r)}$. Similarly, $L_g=\{g\ell_1,\ell'_2,\cdots,\ell'_n\}$ is a filling collection, and there is a homeomorphism $h_2\in \Gamma_{L_g,(r)}$ such that 
		$h_2\ell_1=\gamma_b$ and $h_2\gamma_b=\gamma_a$.
		Now $g_1\defeq h_2^{-1}h_1g$ preserves $\ell_1$ and $g_1\in\Gamma_{(r)}\cap \Gamma^0$, hence $g=(h_1^{-1} h_2)\cdot g_1$ is the desired factorization.
	
	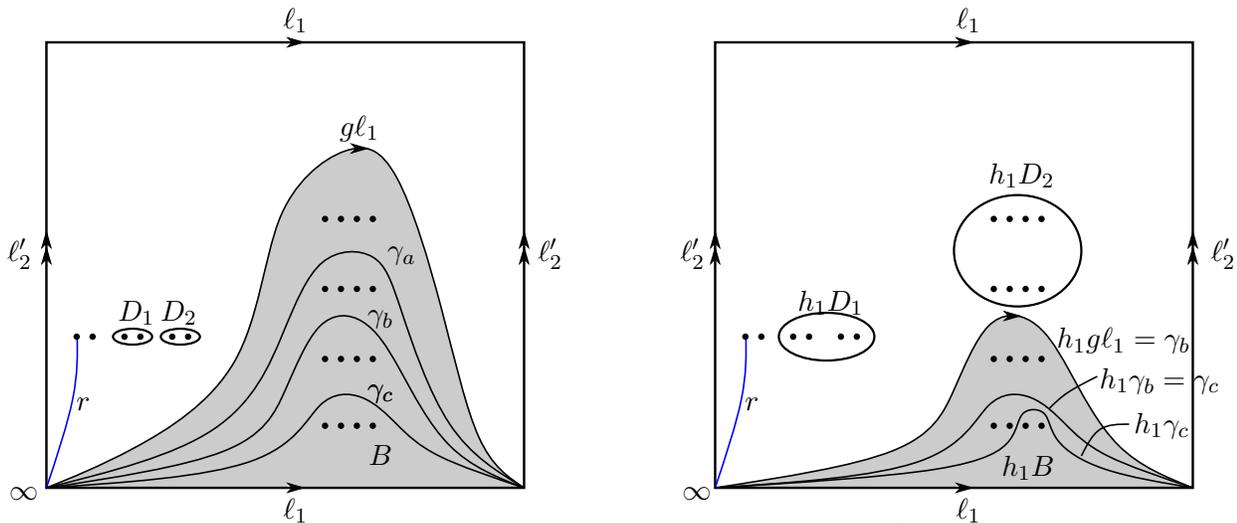
\begin{figure}[hbtp]
		\labellist
		\pinlabel $\infty$ at -8 0
		\pinlabel $r$ at 20 42
		\pinlabel $D_1$ at 45 85
		\pinlabel $D_2$ at 65 85
		\pinlabel $g\ell_1$ at 150 172
		\pinlabel $\gamma_a$ at 170 112
		\pinlabel $\gamma_b$ at 160 82
		\pinlabel $\gamma_c$ at 160 47
		\pinlabel $\ell_1$ at 120 223
		\pinlabel $\ell_1$ at 120 -8
		\pinlabel $\ell'_2$ at -10 113
		\pinlabel $\ell'_2$ at 240 113
		\pinlabel $B$ at 160 18
		
		\pinlabel $\infty$ at 309 0
		\pinlabel $r$ at 335 42
		\pinlabel $h_1D_1$ at 372 90
		\pinlabel $h_1D_2$ at 462 150
		\pinlabel $h_1g\ell_1=\gamma_b$ at 510 72
		\pinlabel $h_1\gamma_b=\gamma_c$ at 527 53
		\pinlabel $h_1\gamma_c$ at 528 32
		\pinlabel $\gamma_c$ at 160 47
		\pinlabel $\ell_1$ at 437 223
		\pinlabel $\ell_1$ at 437 -8
		\pinlabel $\ell'_2$ at 307 113
		\pinlabel $\ell'_2$ at 557 113
		\pinlabel $h_1B$ at 465 12
		\endlabellist
		\centering
		\includegraphics[scale=0.8]{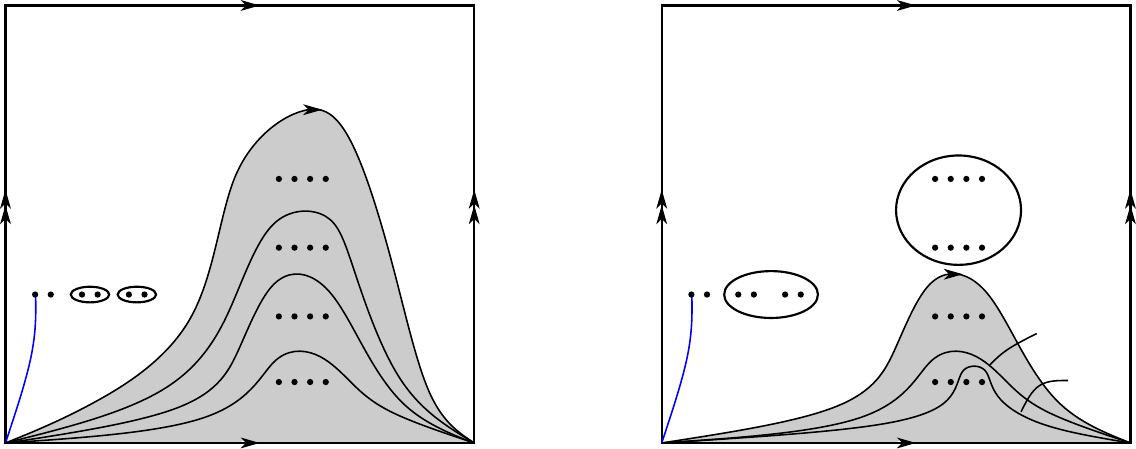}
		\caption{A choice of lassos $\gamma_a$, $\gamma_b$, $\gamma_c$ in the case $S=T^2 -\infty$ and the effect of the homeomorphism $h_1$}
		\label{fig: cutbigon}
	\end{figure}
	
		Suppose $g\ell_1$ and $\ell_1$ intersect. Put them in minimal position as lassos on 
		$\Sigma$. Then they intersect finitely many times, and there is an innermost disk $B$ 
		in $S$ which is a bigon bounded by arcs $\alpha\subset \ell_1$ and $\alpha_g\subset g\ell_1$. 
		Then $B\cap K$ is nonempty. Note that if we replace $\alpha_g$ by $\alpha$ to modify 
		$g\ell_1$ into a lasso $\ell'_1$, then $\ell'_1$ has a smaller intersection number with $\ell_1$. 
		
		We can decompose $B$ into four parts as in the previous case, and obtain a 
		factorization $g=h'g'$, where $g'\in \Gamma_{(r)}\cap \Gamma^0$ has the 
		property that $g'\ell_1=\ell'_1$, and $h'$ is a product of two mapping classes 
		in $\Gamma_{(r)}$ that each fix a filling collection with respect to $r$. 
		Since $g'\ell_1=\ell'_1$ has a smaller intersection number with $\ell_1$, by induction 
		we can reduce to the disjoint case treated earlier and complete the proof of Claim \ref{claim: factorization}.
	\end{proof}
	
	We continue the proof of Lemma~\ref{lemma: vary L}.
	Now, $g_1$ fixes $\ell_1$. Furthermore, both $g_1\ell_2$ and $\ell_2$ are disjoint from $\ell_1=g\ell_1$, 
	and arguing as in the proof of Claim \ref{claim: factorization} we obtain a factorization $g_1=h_2g_2$, 
	where $g_2\in \Gamma_{(r)}\cap \Gamma^0$ fixes both $\ell_1$ and $\ell_2$, and $h_2$ is 
	the product of finitely many mapping classes in $\Gamma_{(r)}$, each fixing a filling 
	collection (containing $\ell_1$) with respect to $r$. Continuing inductively, 
	we obtain a factorization $g=h_1h_2\cdots h_n g_n$ where $g_n\in \Gamma_{(r)}\cap \Gamma^0$ fixes each $\ell_i$, and each $h_i$ is a product of finitely many mapping classes in $\Gamma_{(r)}$ that each fixes a filling collection with respect to $r$,
	where $1\le i\le n$. Hence $g_n\in \Gamma_{L_0,(r)}$, which completes the proof.
\end{proof}

In what follows, whenever we say two short rays are disjoint we require them to have distinct endpoints in $K$.
\begin{lemma}\label{lemma: transitive on disjoint rays}
	For any short ray $r$, the group $\Gamma_r\cap \Gamma^0$ acts transitively on the set of
	short rays disjoint from $r$. 
\end{lemma}
\begin{proof}
	Let $s$ and $s'$ be short rays disjoint from $r$. If $s$ and $s'$ are also disjoint, then there is a closed disk $D$ containing $\infty$ that is the union of disk neighborhoods of $r$, $s$ and $s'$, 
	so that $\partial D$ does not intersect the $K$ and $D\cap K$ is a Cantor set. 
	Then there is a homeomorphism $h$ supported in $D$ fixing $r$ and taking $s$ to $s'$.
	Thus $h$ represents a mapping class $g\in \Gamma_r\cap \Gamma^0$ with $gs=s'$.
	
	For the general case, we can choose a short ray $r'$ very close to $r$ such that $r'$ is disjoint from $r,s,s'$. Then by the above argument, there are $g,g'\in \Gamma_r\cap \Gamma^0$ such that $gs=r'$ and $g's'=r'$. Hence $g'^{-1}g\in \Gamma_r\cap \Gamma^0$ takes $s$ to $s'$.
\end{proof}

\begin{lemma}\label{lemma: (r) to r}
	For any disjoint short rays $r$ and $s$, the group $\Gamma_r\cap \Gamma^0$ is generated by $\Gamma_{(r)}\cap \Gamma^0$ and $\Gamma_{(s)}\cap\Gamma_r\cap \Gamma^0$.
\end{lemma}
\begin{proof}
	Let $g$ be a homeomorphism fixing $r$ representing a mapping class in $\Gamma_r\cap \Gamma^0$. 
	Then $g=g_1g_2$ where $g_2$ is supported in an arbitrarily small disk neighborhood $D$ of the endpoint of $r$ and $g_1$ is the identity on $D\cup r$. Hence $g_1$ represents a mapping class in $\Gamma_{(s)}\cap\Gamma_r\cap \Gamma^0$ by choosing $D$ small enough and $g_1$ represents a mapping class in $\Gamma_{(r)}\cap \Gamma^0$.
\end{proof}

\begin{lemma}\label{lemma: generate Gamma^0}
	Let $r$ and $s$ be disjoint short rays. Then $\Gamma_r\cap \Gamma^0$ and $\Gamma_s\cap \Gamma^0$ generate $\Gamma^0$.
\end{lemma}
\begin{proof}
	Let $\Gamma^0(r,s)$ be the group generated by $\Gamma_r\cap \Gamma^0$ and $\Gamma_s\cap \Gamma^0$. 
	It suffices to show that $\Gamma^0(r,s)$ acts transitively on the set of short rays. Here is
	why. For any $\gamma\in \Gamma^0$, transitivity implies that there is 
	$g\in \Gamma^0(r,s)$ such that $gr=\gamma r$ and thus $\gamma=gh$ where 
	$h\defeq g^{-1}\gamma\in \Gamma_r\cap \Gamma^0\subset \Gamma^0(r,s)$.
	
	Now we prove transitivity.
	We show for any short ray $x$ there is some $g\in \Gamma^0(r,s)$ with $gx=s$ by induction 
	on the distance between $x$ and $s$ in the \emph{ray graph} $\Ray$. 
	The vertex set of $\Ray$ is the set of short rays, and two vertices are connected by an edge if they are represented by disjoint short rays. 
	The ray graph is connected: any two short rays with distinct endpoints have a finite intersection number, and the result follows from a classical argument proving connectivity of curve graphs of surfaces of finite type; 
	see the proof of \cite[Theorem 4.3]{Farb_Margalit} for an example.
	
	Hence there is a geodesic path in the ray graph $\Ray$ connecting $x$ and $s$, 
	i.e. there is a minimal integer $n$ and short rays $x_0=x, x_1,\cdots, x_n=s$ such that adjacent short rays are disjoint. 
	For the base case $n=1$ where $x$ and $s$ are disjoint, 
	since $\Gamma_s\cap \Gamma^0$ acts transitively on rays disjoint from $s$ by Lemma \ref{lemma: transitive on disjoint rays}, 
	there is some $h\in \Gamma_s\cap \Gamma^0$ that takes $x$ to a short ray disjoint from both $r$ and $s$,  which can be further taken to $s$ by an element of $\Gamma_r\cap \Gamma^0$.
	
	Suppose $n\ge2$ and there is some $g\in \Gamma^0(r,s)$ taking $x_1$ to $x_n=s$. Then $gx$ is disjoint from $gx_1=s$ and by the base case there is some $g'\in \Gamma^0(r,s)$ taking $gx$ to $s$. Hence $g'gx=s$ as desired. This completes the proof.
\end{proof}

\begin{prop}\label{prop: normal closure large}
	The group $\Gamma^0=\Mod(\Sigma)_0$ is generated by mapping classes supported in the dividing disks.
	In other words, $\Ndiv = \Gamma^0$.
\end{prop}
\begin{proof}
	Evidently $\Ndiv \subset \Gamma^0$ since each mapping class supported in a disk has trivial
	image in $\Mod(S)$. So we just need to show $\Gamma^0 \subset \Ndiv$.
	For any short ray $r$ and any collection $L$ of lassos that are disjoint from $r$ and cut $S$ into a disk, the subgroup $\Gamma_{L,(r)}$ is a subgroup of $\Ndiv$ by Lemma \ref{lemma: GammaL(r) contained in N}. 
	Thus $\Ndiv$ contains $\Gamma_{r}\cap \Gamma^0$ by Lemmas \ref{lemma: vary L} and \ref{lemma: (r) to r} for any short ray $r$. 
	Then it follows from Lemma \ref{lemma: generate Gamma^0} that $\Ndiv$ contains $\Gamma^0$.
\end{proof}

From this we can deduce the Purity Theorem~\ref{theorem:pure} when $S$ has at most one puncture.
\begin{proof}[Proof of Theorem~\ref{theorem:pure} when $S$ has at most one puncture]
	By Lemma \ref{lemma: closed case}, it suffices to consider the case where $S$ 
	is a closed surface minus one puncture. 
	By Lemma \ref{lemma: normal act trivially} and 
	Proposition \ref{prop: normal closure large}, the normal closure of any 
	$g\in\Gamma\setminus\PMod(\Sigma)$ contains $\Ndiv=\Gamma^0$. 
	Hence any normal subgroup not contained in $\PMod(\Sigma)$ must contain $\Gamma^0$.
\end{proof}

\subsection{Proof of the Purity Theorem in general}\label{subsec: Purity general case}
In this section we prove the Purity Theorem \ref{theorem:pure} in general. Let $S$ be a surface of finite genus with finitely many punctures $P$. Let $K$ be a Cantor subset of $S$. Denote $\Sigma=S\setminus K$.

Note that any mapping class in $\Mod(\Sigma)=\Mod(S,K)$ preserves the sets $P$ and $K$ respectively.
The subgroup $\PMod(\Sigma,P)$ fixing $P$ pointwise is a proper subgroup if and only if $P$ has more than one puncture.

The proof in Section \ref{subsec: Purity one puncture} without many changes proves the following version of the Purity Theorem for $\PMod(\Sigma,P)$, from which we will deduce the Purity Theorem for $\Mod(S,K)$.
\begin{thm}\label{thm: Purity for pure subgroup}
	Let $S$ be a surface of finite genus with finitely many punctures $P$.
	Let $K$ be a Cantor subset of $S$.
	Then any normal subgroup $N$ of $\PMod(S\setminus K,P)$ either contains the kernel of the forgetful map to $\PMod(S,P)$, or it is contained in $\PMod(S,K\cup P)$.
\end{thm}
\begin{proof}
	Let $\Sigma=S\setminus K$.
	When $P$ has at most one element, we have $\PMod(\Sigma,P)=\Mod(\Sigma)=\Mod(S,K)$, $\PMod(S,P)=\Mod(S)$, and $\PMod(S,K\cup P)=\PMod(S,K)$. Thus this is the Purity Theorem we proved in Section \ref{subsec: Purity one puncture}.
	
	The proof strategy still works when $P$ has more than one element, which we explain as follows. 
	As in Section \ref{subsec: Purity one puncture}, let $\Gamma=\PMod(\Sigma,P)$, let $\Gamma^0$ be the kernel of the forgetful map to $\PMod(S,P)$ and let $\Ndiv$ denote the subgroup of $\Gamma$ generated by elements supported in dividing disks. Then Lemma \ref{lemma: normal act trivially} using the same proof holds for any $g\in \Gamma\setminus \PMod(S,K\cup P)$. So it remains to show that $\Ndiv=\Gamma^0$.
	
	Fix a puncture $\infty \in P$.
	Define short rays and lassos as before using the chosen puncture $\infty$. The only difference in the proof is the definition of a filling collection $L$: Here for a given short ray $r$, we say a collection $L$ is \emph{filling} (with respect to $r$) if it consists of disjoint lassos and $|P|-1$ arcs in $\Sigma$ each connecting $\infty$ to a distinct element of $P\setminus\{\infty\}$ so that they are all disjoint from $r$ and cut $S$ into a single disk.
	
	With this modified definition, the analogs of Lemmas \ref{lemma: GammaL(r) contained in N}, \ref{lemma: vary L} , \ref{lemma: transitive on disjoint rays}, \ref{lemma: (r) to r}, and \ref{lemma: generate Gamma^0} can be proved in the same way, and they imply $\Ndiv=\Gamma^0$ as desired.
\end{proof}

Now we can deduce Purity Theorem \ref{theorem:pure} in full generality.
\begin{proof}[Proof of Purity Theorem \ref{theorem:pure}]
	Let $\Sigma=S\setminus K$ as above. Note that the forgetful map $\pi:\Mod(\Sigma)\to \Mod(S)$ restricts to the forgetful map from $\PMod(\Sigma,P)$ to $\PMod(S,P)$, and that the two maps have the same kernel since any mapping class in $\Mod(\Sigma)$ that acts trivially on $S$ must preserve $P$ pointwise.
	
	First consider the case where the normal subgroup $N$ of $\Mod(S,K)$ contains an element $g$ that acts nontrivially on $K$ and preserves $P$ pointwise. Then the normal subgroup $N\cap \PMod(\Sigma,P)$ of $\PMod(\Sigma,P)$ is not contained in $\PMod(S,K\cup P)$. Thus by Theorem \ref{thm: Purity for pure subgroup}, $N\cap\PMod(\Sigma,P)$ contains the kernel $\ker \pi$, and so does $N$.
	
	In general, if a normal subgroup $N$ of $\Mod(S,K)$ contains an element $g$ that acts nontrivially on $K$, then the image of $N$ in $\Aut(K)$ is a nontrivial normal subgroup since $\Mod(S,K)\to \Aut(K)$ is surjective. As $\Aut(K)$ is simple by Anderson's theorem \cite{Anderson} (which also follows from our proof of Purity Theorem for $S=\R^2$), we may choose $g$ so that its image in $\Aut(K)$ has infinite order. Then a power of $g$ fixes $P$ pointwise and it acts nontrivially on $K$. Thus $N$ must contain $\ker \pi$ by the previous paragraph. This completes the proof.
\end{proof}

\section{The Inertia Theorem}

In this section we let $S$ be {\em any} connected oriented surface (finite type or not).
Recall that for any compact totally disconnected subset $Q$ of $S$ the mapping class group of
$S$ rel. $Q$ is denoted $\Mod(S,Q)$. This is a subgroup of $\Mod(\Sigma)$ where $\Sigma = S-Q$
but in general it might be smaller, since a typical element of $\Mod(\Sigma)$ might permute ends
of $S$ with points of $Q$. Similarly, $\PMod(S,Q)$ denotes the {\em pure} subgroup of $\Mod(S,Q)$; 
i.e.\/ the subgroup of mapping classes fixing $Q$ pointwise. 
Throughout this section, we say a (normal) subgroup of $\Mod(S,Q)$ 
is pure (normal) if it is contained in $\PMod(S,Q)$.

For $Q$ finite, the groups $\Mod(S,Q)$ and $\PMod(S,Q)$ 
depend up to isomorphism only on the cardinality of $Q$, 
and by abuse of notation we fix representatives of these
isomorphism classes which we denote $\Mod(S,n)$ and $\PMod(S,n)$. Each $\Mod(S,Q)$ is
isomorphic to $\Mod(S,n)$ and each $\PMod(S,Q)$ is isomorphic to $\PMod(S,n)$ by an isomorphism
(in either case) unique up to an {\em inner automorphism} (in $\Mod(S,n)$). Thus there is a
{\em canonical} bijection between normal subgroups of any $\Mod(S,Q)$ contained in $\PMod(S,Q)$
and normal subgroups of $\Mod(S,n)$ contained in $\PMod(S,n)$.

\begin{defn}[Spectrum]
Let $K \subset S$ be a Cantor set, and let $N$ be a normal subgroup of $\Mod(S,K)$ contained in
$\PMod(S,K)$. For any natural number $n$, let $Q \subset K$ be an $n$-element set. The image
of $N$ under the forgetful map $\PMod(S,K) \to \PMod(S,Q)$ determines a normal subgroup 
$N_n$ of $\Mod(S,n)$ contained in $\PMod(S,n)$ depending only on $n$ (and not on $Q$). 
The {\em spectrum} of $N$ is the sequence $\normseq$ of pure normal subgroups $N_n \subset \Mod(S,n)$.
\end{defn}

The spectrum does not determine $N$; we shall see some examples of different groups with
the same spectrum in Example~\ref{example:different_same_spectrum}.

Each group $N_n$ determines $N_m$ for any $m < n$, since $N_m$ is just the image of $N_n$ under
any forgetful map $\PMod(S,n) \to \PMod(S,m)$ that forgets $n-m$ marked points. This already
imposes nontrivial conditions on the $N_n$ that we formalize as follows:

\begin{defn}[Algebraically Inert]
A family (indexed by $n \in \N$) of normal subgroups 
$\normseq$ of $\Mod(S,n)$, each $N_n$ contained in $\PMod(S,n)$, 
is called an {\em algebraically inert family} if for each $n>m$ the image of $N_n$ in $\PMod(S,m)$ 
under each of the $n$ choose $m$ forgetful maps $\PMod(S,n) \to \PMod(S,m)$ is equal to $N_m$.

A normal subgroup $N_n$ is called {\em algebraically inert} if it is a member of some
algebraically inert family.
\end{defn}

If $\normseq$ is the spectrum of $N$ then necessarily $\normseq$ is an algebraically inert family 
and every $N_n$ is algebraically inert. 
We do not know if every algebraically inert family arises as the spectrum of some $N$, or
even if each individual algebraically inert $N_n$ arises in the spectrum of some $N$.
Nevertheless we {\em are} able to give a complete characterization of which subgroups 
of $\PMod(S,n)$ arise as $N_n$ for some normal subgroup $N$ of $\PMod(S,K)$. 
This is the main result of the section, the {\em Inertia Theorem}~\ref{theorem:inert}
and such subgroups $N_n$ are said to be {\em inert} (see Definition~\ref{definition:inert}).

\subsection{Inert subgroups}\label{subsection:inertia}

Let $S$ be any surface and let $Q \subset S$
be a finite subset with cardinality $n$. We shall describe an operation on mapping classes
in $\PMod(S,Q)$ called {\em insertion}. Informally, this operation takes as input
a pure mapping class $\alpha$, chooses a representative homeomorphism $\tilde{\alpha}$ 
which is equal to the identity on a neighborhood $D_Q$ of $Q$, and taking the conjugacy
class of $\tilde{\alpha}$ rel. $Q'$ where $Q'$ is contained in $D_Q$ and $Q'$ is finite
with $|Q'|=|Q|$.

Now let's make this more precise. Let $D_Q \subset S$ be a collection of $n$
disjoint closed disks, each centered at some point of $Q$ and let $\pi:D_Q \to Q$ be the 
retraction that takes each disk to its center. By abuse of notation, let 
$\PMod(S,D_Q)$ denote the mapping class group of $S$ fixing $D_Q$ {\em pointwise}.
There is a surjective forgetful map $\PMod(S,D_Q) \to \PMod(S,Q)$ and, as is well-known, 
this is a $\Z^n$ central extension generated by Dehn twists around the boundary components
of $D_Q$ (see e.g.\/ Farb--Margalit \cite[Prop~3.19]{Farb_Margalit}; 
they call this the {\em capping homomorphism}). 

\begin{defn}[Insertion]
Let $\alpha \in \PMod(S,Q)$ be any element, and let $\hat{\alpha}$ be some preimage in
$\PMod(S,D_Q)$. Let $f:Q \to Q$ be any map (possibly not injective) and $n=|f(Q)|$. The {\em insertion}
$f^*\hat{\alpha} \subset \PMod(S,n)$ is the ($\Mod(S,n)$)-conjugacy class of the element defined as
follows. Let $\tilde{f}:Q \to D_Q$ be any injective map for which the composition $\pi\tilde{f}=f$.
Then $f^*\hat{\alpha}$ is represented by the image of $\hat{\alpha}$ under the forgetful map to 
$\PMod(S,f(Q))$ which may be
canonically identified (up to conjugacy in $\Mod(S,n)$) with $\PMod(S,n)$.
\end{defn}

 Note that the
class of $f^*\hat{\alpha}$ is invariant under right composition of $f$ with a permutation of $Q$,
and it depends only on the cardinalities of the preimages (under $f$) of the elements of $Q$.

\begin{defn}[Inert]\label{definition:inert}
A subgroup $N \subset \PMod(S,n)$ normal in $\Mod(S,n)$ is said to be {\em inert} if,
after fixing some identification of $\PMod(S,n)$ with $\PMod(S,Q)$, for every
$\alpha \in N$ there is a lift $\hat{\alpha}$ in $\PMod(S,D_Q)$ so that for every
$f:Q \to Q$ the insertion $f^*\hat{\alpha}$ is in $N$.
\end{defn}

Given $N$ we can define $\hat{N} \subset \PMod(S,D_Q)$ to be the collection of all
$\hat{\alpha}$ lifting some $\alpha \in N$ for which every insertion $f^*\hat{\alpha}$ is in $N$.
We call $\hat{N}$ the {\em corona} of $N$. By the definition of inert, the corona $\hat{N}$
surjects to $N$.

\begin{lemma}\label{lemma:corona_is_group}
If $N \subset \PMod(S,n)$ is inert, the corona $\hat{N}$ is a subgroup
of $\PMod(S,D_Q)$.
\end{lemma}
\begin{proof}
For any $f$ and any $\alpha,\beta$ with lifts $\hat{\alpha},\hat{\beta}$ there are
identities $f^*\hat{\alpha}^{-1} = (f^*\hat{\alpha})^{-1}$ and 
$f^*\hat{\alpha} f^*\hat{\beta} = f^* (\hat{\alpha}\hat{\beta})$. The proof follows.
\end{proof}

The following lemma allows one to construct many inert subgroups:

\begin{lemma}\label{lemma:subset_generates_inert}
Let $X \subset \PMod(S,Q)$ be a {\em subset} closed under insertion. In other words, for every
$\alpha \in X$, there is $\hat{\alpha} \in \PMod(S,D_Q)$ so that for any $f:Q \to Q$
the insertion $f^*\hat{\alpha}$ lies in $X$. Then the subgroup of $\PMod(S,Q)$ generated 
by $X$ is inert.
\end{lemma}
\begin{proof}
Note that for any $f$ and any $\alpha \in \PMod(S,Q)$ there is an identity 
$f^*\hat{\alpha}^{-1} = (f^*\hat{\alpha})^{-1}$. So without loss of generality we may
assume $X$ is closed under taking inverses.

Now, let $G$ be the subgroup of $\PMod(S,Q)$ generated by $X$, and let $x \in G$ be arbitrary, so 
that $x=x_1\cdots x_n$ where every $x_i \in X$. For each $i$ let $\hat{x}_i \in \PMod(S,D_Q)$
be the lift promised by the definition of insertion. If we define
$\hat{x}:=\hat{x}_1 \cdots \hat{x}_n$ then evidently $\hat{x}$ is a lift of $x$ to $\PMod(S,D_Q)$, and
$$f^* \hat{x} = f^* \hat{x}_1 \cdots \hat{x}_n = (f^*\hat{x}_1) \cdots (f^*\hat{x}_n) \in G.$$
\end{proof}

\begin{exmp}[$n$th powers of Dehn twists]\label{example:inert_powers}
For any $n$ the subgroup of $\PMod(S,Q)$ generated by $n$th powers of Dehn twists along embedded loops 
is inert. For, the {\em set} of $n$th powers of Dehn twists along embedded loops is closed under
insertion. Now apply Lemma~\ref{lemma:subset_generates_inert}.
\end{exmp}

\begin{exmp}[Inert-Brunnian subgroup]\label{example:brunnian}
For any $n$ define the {\em Inert-Brunnian subgroup} of $\PMod(S,Q)$ to
be the group of pure mapping classes $\alpha \in \PMod(S,Q)$ that have lifts 
$\hat{\alpha} \in \PMod(S,D_Q)$ for which every nontrivial insertion (i.e.\/
every insertion other than a permutation) is the identity. Note that every such mapping class
is Brunnian in the usual sense.
\end{exmp}

\begin{exmp}[$\PSL(2,\Z)$]
Let $S=\R^2$. Then $\Mod(\R^2,3)=\PSL(2,\Z)$ and $\PMod(\R^2,3)$ is a free rank 2 subgroup of
index 6. Since $\PMod(\R^2,2)$ is trivial it follows that {\em every} normal subgroup of
$\PSL(2,\Z)$ contained in $\PMod(\R^2,3)$ is inert.
\end{exmp}

The main theorem of this section is the {\em Inertia Theorem}: 

\begin{thm}[Inertia Theorem]\label{theorem:inert}
Let $S$ be any connected, orientable surface, and let $K$ be a Cantor set in $S$. 
A subgroup $N_n$ of $\PMod(S,n)$ is equal to the image of some $\Mod(S,K)$-normal
pure subgroup $N \subset \PMod(S,K)$ under the forgetful map if and only if it is inert.
\end{thm}

\begin{rmk}
Insertion bears a family resemblance to Boyland's {\em forcing} order on braids \cite{Boyland}.
It is possible to restate the definitions of inertia and algebraic inertia in the language of operads and FI-modules; 
see e.g.\/ \cite{Church_Putman}. One wonders whether such a reformulation could lead to 
sharper insights into the structure of such groups.
\end{rmk}

\subsection{Proof of the Inertia Theorem}

The goal of this section is to prove the Inertia Theorem~\ref{theorem:inert}.
One direction is easy: given an inert group $N_n \subset \PMod(S,n)$ we construct a pure
normal $N$ with $N_n$ in the spectrum.

\begin{lemma}\label{lemma:inert_is_realized}
For every inert subgroup $H \subset \PMod(S,n)$ there is a pure normal subgroup $N$ of $\Mod(S,K)$
whose restriction to every $n$-element subset of $K$ is conjugate to $H$.
\end{lemma}
\begin{proof}
Let $D_n$ be a family of $n$ disjoint disks in $S$ and let $\hat{H} \subset \PMod(S,D_n)$ 
be the corona of $H$. For every family $D_Q$ of $n$ disjoint disks in $S$ whose interior
contains $K$ and each of whose components intersects $K$, we can fix a homeomorphism of
$S$ taking $D_n$ to $D_Q$ and therefore an isomorphism
$\PMod(S,D_n) \to \PMod(S,D_Q)$ taking $\hat{H}$ to $\hat{H}_Q$. 
Under the forgetful map $\PMod(S,D_Q) \to \PMod(S,K)$ the image of $\hat{H}_Q$ is realized
by homeomorphisms fixed pointwise on $D_Q$. Let $N$ be the subgroup of $\PMod(S,K)$ generated
by all the images of $\hat{H}_Q$'s. Then by the definition of inert, the restriction of $N$ to every 
$n$-element subset of $K$ is conjugate to $H$.
\end{proof}

We shall have cause to refer to the group $N$ constructed from $H$ in 
Lemma~\ref{lemma:inert_is_realized}. We call it the {\em flattening} of $H$, and denote it
$N^\flat(H)$. We can use it to give a simple example of distinct normal subgroups in
$\PMod(S,K)$ with the same spectrum.

\begin{lemma}\label{lemma:flat_property}
Let $N$ be equal to $N^\flat(H)$ for some $H$. 
Then for every $\alpha \in N^\flat(H)$
there is a neighborhood $V$ of $K$ in $S$ so that $\alpha$ is represented by a homeomorphism
fixed pointwise on $V$.
\end{lemma}
\begin{proof}
This property holds by definition for the generators (by taking $V=D_Q$), and is preserved under finite products.
\end{proof}

\begin{exmp}\label{example:different_same_spectrum}
Let $N_\text{Dehn}$ be the subgroup of $\PMod(S,K)$ generated by all Dehn twists. Then
$N_\text{Dehn}:=N^\flat(\PMod(S,n))$ for every $n$ and therefore its spectrum is equal to
precisely the sequence $\{\PMod(S,n)\}$. On the other hand, the entire group
$\PMod(S,K)$ has the same spectrum as $N_\text{Dehn}$ but fails to have the property of flattened
subgroups promised by Lemma~\ref{lemma:flat_property}. 
When $S$ is closed, $N_\text{Dehn}$ is the subgroup of compactly supported mapping classes and its closure under the compact-open topology
is exactly $\PMod(S,K)$ by \cite[Theorem 4]{Patel_Vlamis}, 
which also explains why $N_\text{Dehn}$ and $\PMod(S,K)$ have the same spectrum in view of Proposition \ref{prop:closure_spectrum}.

Here is another example. For any $p$ let $N_{\text{Dehn}^p}$ be the group generated by $p$th powers
of all Dehn twists; evidently this group is a flattening of any of the inert subgroups from
Example~\ref{example:inert_powers}. On the other hand, let $N'_p$ be the group generated by
simultaneous $\pm p$th powers of Dehn twists in arbitrary (possibly infinite) 
families of disjoint curves. Then $N'_p$ and $N_{\text{Dehn}^p}$ have the same spectrum but
$N'_p$ fails to have the property promised by Lemma~\ref{lemma:flat_property}.
\end{exmp}

The harder direction in the proof of the Inertia Theorem is to show that for every 
pure normal $N$, every $N_n$ in the spectrum is inert. In other words:

\begin{lemma}\label{lemma:restriction_is_inert}
Let $N$ be a pure normal subgroup of $\Mod(S,K)$. Then every $N_n$ in the spectrum of $N$ is
inert.
\end{lemma}

Lemma~\ref{lemma:restriction_is_inert} will follow from Proposition~\ref{proposition:local_behaviour}
which might be of independent interest.

\begin{prop}\label{proposition:local_behaviour}
	For any $\Mod(S,K)$-normal subgroup $N$ of $\PMod(S,K)$, given any element $\gamma\in N$, 
	any positive integer $M$ and any $Q:=\{q_1,\cdots,q_n\}$ an $n$-element subset of $K$ for some $n$,
	there is another element $\gamma'\in N$, a neighborhood $D_Q:=\sqcup_{i=1}^n D_i$ of $Q$ consisting of disjoint dividing disks $D_i$
	and a finite subset $X=\sqcup_{i=1}^n X_i$ where each 
	$X_i\subset K\cap D_i$ has size at least $M$ with $q_i\in X_i$, such that 
	\begin{enumerate}
	\item there exists $\gamma_D \in\PMod(S,D_Q)$ (which will play the role of $\hat{\alpha}$ in the definition of inert subgroups) 
	whose image under the forgetful map $\PMod(S,D_Q)\to \PMod(S,X)$ 
	is the same as the image of $\gamma'$ under the forgetful map $\PMod(S,K)\to\PMod(S,X)$; and
	\item $\gamma'$ and $\gamma$ have the same restriction to $\PMod(S,Q)$.
	\end{enumerate}
\end{prop}

We prove Lemma~\ref{lemma:restriction_is_inert} assuming Proposition~\ref{proposition:local_behaviour}.
\begin{proof}[Proof of Lemma~\ref{lemma:restriction_is_inert}]
For any $\alpha \in N_n$, we choose an $n$-element set $Q:=\{q_1,\cdots,q_n\} \subset K$
and by abuse of notation we identify $\alpha$ with a conjugacy class $\alpha_Q$ in $\PMod(S,Q)$.
Let $\gamma \in N$ restrict to $\alpha_Q$ in $\PMod(S,Q)$.
Choose $M\ge n$. 
Let the following objects be as promised in Proposition~\ref{proposition:local_behaviour}: $\gamma'\in N$, $\gamma_D \in \PMod(S,D_Q)$, a neighborhood $D_Q=\sqcup_{i=1}^n D_i$ consisting of disjoint dividing disks and subsets $X_i\subset D_i\cap K$ with each $|X_i|\ge n$. If we think of $\PMod(S,D_Q)$
as a central extension of $\PMod(S,Q)$, then $\gamma_D$ is a lift of $\alpha_Q$. Furthermore
for every map $f:Q \to Q$ we can realize $f$ by an injective map $\tilde{f}:Q \to X$ projecting to
$f$ under the obvious projection $X \to Q$ that maps $X_i$ to $q_i$, and
then the restriction of $\gamma'$ to $\tilde{f}(Q)$ is equal to the restriction of $\gamma_D$ to
$\tilde{f}(Q)$ which is (by definition) equal to the insertion $f^*(\gamma_D)$.

Since $\gamma' \in N$ it follows that the conjugacy class of the insertion $f^*(\gamma_D)$ is in $N_n$
for every $f$. Thus $N_n$ is inert, as desired.
\end{proof}

It remains to prove Proposition~\ref{proposition:local_behaviour}, for which we need the following lemma.


\begin{lemma}\label{lemma:nested_Dehn_twists}
	Let $D$ be a closed disk and $K$ be a Cantor set in the interior of $D$. Fix a pure mapping class $\gamma\in \PMod(D,K)$ and $x\in K$.
	Then there is an infinite sequence of distinct points $X=(x_0,x_1,\cdots)$ in $K$ with $x_0=x$ and a sequence of nesting dividing disks $D_1\supset D_2\supset\cdots$ converging to some $x_\infty\in K$ such that
	\begin{enumerate}
		\item $x_j\in (D_j\setminus D_{j+1})\cap K$ for all $j\ge0$, where $D_0=D$, and
		\item the image of $\gamma$ in $\PMod(D,X)$ is represented by a homeomorphism $h=\prod_{j\ge0}  t_j^{m_j}$,
		where $m_j\in\Z$ and $t_j$ is a (counter-clockwise) Dehn twist around $\partial D_j$ supported in a small tubular neighborhood $T_j$ of $\partial D_j$ in $D_j$ away from $x_j$ and $D_{j+1}$.
	\end{enumerate}
\end{lemma}
\begin{proof}
	Let $x_1\neq x_0\in K$ be an arbitrary point. Then $\PMod(D,\{x_0,x_1\})\cong \Z$ is generated by a (counter-clockwise)  Dehn twist, 
	which we represent by a homeomorphism $t_0$ supported in a tubular neighborhood $T_0$ of $\partial D_0$ in $D_0=D$ so that $K\cap T_0=\emptyset$.
	This determines an integer $m_0$ such that $t_0^{m_0}$ represents the image of $\gamma$ in $\PMod(D,\{x_0,x_1\})$.
	
	Represent $\gamma$ by a homeomorphism $\phi$ in $D$. Choose any dividing disk $D_1$ with $x_1\in D_1$ and $x_0\notin D_1$, clearly $\phi(D_1)$ is isotopic to $D_1$ in $(D,\{x_1\}\sqcup (K\setminus D_1))$ as they both shrink to $x_1$.
	By continuity, for any Cantor set $K_1\subset K\cap D_1$ sufficiently close to $x_1$ so that $\partial D_1$ stays disjoint from it under the isotopy, we have $\phi(D_1)$ isotopic to $D_1$ in $(D,K_1 \sqcup (K\setminus D_1))$.
	In particular, the image $\gamma_1$ of $\gamma$ in $\PMod(D,K_1\sqcup \{x_0\})$ preserves the disk $D_1$ (up to isotopy).
	Thus we can represent $\gamma_1$ as a homeomorphism on $D$ which is the Dehn twist $t_0^{m_0}$ together with a homeomorphism $\phi_1$ supported in $D_1$ fixing its boundary and $K_1$ pointwise.
	
	Now choose an arbitrary $x_2\neq x_1$ in $K_1$ and represent the (counter-clockwise) Dehn twist generating $\PMod(D_1,\{x_1,x_2\})\cong \Z$ by a homeomorphism $t_1$ 
	supported in a tubular neighborhood $T_1$ of $\partial D_1$ in $D_1$ so that $K_1\cap T_1=\emptyset$. Then $\phi_1$ as a mapping class in $\PMod(D_1,\{x_1,x_2\})$ agrees with $t_1^{m_1}$ for some $m_1\in \Z$.
	
	Repeat the process above inductively. 
	Since $D_j$ and $x_j$ are quite arbitrarily chosen (as long as $x_j$ is sufficiently close to $x_{j-1}$ so that it lies in $K_j$), we can make sure that the sequence $\{x_j\}$ converges to some $x_\infty \in K$ and $\cap_{j=0}^\infty D_j=\{x_\infty\}$.
	Then for $X=(x_0,x_1,\cdots)$, the image of $\gamma$ in $\PMod(D,X)$ has the desired property.
\end{proof}

\begin{proof}[Proof of Proposition~\ref{proposition:local_behaviour}]
	Let $D_Q:=\sqcup_{i=1}^n D_i$ be a neighborhood of $Q$ consisting of disjoint dividing disks $D_i$ with $q_i\in D_i$. 
	As in the proof of Lemma~\ref{lemma:nested_Dehn_twists}, by continuity, there are Cantor sets $K'_i\subset K\cap D_i$ for all $i$ such that
	the image $\bar{\gamma}$ of $\gamma$ in $\PMod(S,K')$ preserves each $D_i$ up to isotopy, where $K'\defeq \sqcup K'_i$.
	Denote the image of $N$ in $\PMod(S,K')$ as $\overline{N}$, which is normal in $\Mod(S,K')$. 
	
	We show below that there is an element $\bar{\gamma}'\in\overline{N}$ and finite sets $X_i\subset K'_i$ with $q_i\in X_i$ (and $|X_i|\ge M$) such that
	the image of $\bar{\gamma}'$ under the forgetful map $\PMod(S,K')\to\PMod(S,\sqcup X_i)$ has the desired property.
	The original assertion follows from this by taking an arbitrary $\gamma'\in N$ which maps to $\bar{\gamma}'$ under the forgetful map $\PMod(S,K)\to\PMod(S,K')$.
	
	Now it remains to construct some element $\bar{\gamma}'\in\overline{N}$ and finite sets $X_i$'s such that
	\begin{enumerate}
		\item $\bar{\gamma}'|_Q=\bar{\gamma}|_Q(=\gamma|_Q)\in \PMod(S,Q)$,\label{item: restrict to x}
		\item $\bar{\gamma}'$ preserves each disk $D_i$ up to isotopy in $(S,\sqcup X_i)$, and\label{item: preserve disks}
		\item $\bar{\gamma}'$ restricts to the identity in each $\Mod(\intrm(D_i),X_i)$.\label{item: restrict to Dehn twists}
	\end{enumerate}
	We will construct $\bar{\gamma}'$ as a product of conjugates of $\bar{\gamma}$ by mapping classes in $\Mod(S,K')$, 
	where each mapping class is represented by a product of homeomorphisms $g_i$, $i=1,\cdots, n$ such that $g_i$ is supported in $D_i$ and preserves $K'_i$.
	In particular, bullet (\ref{item: preserve disks}) will follow automatically from the nature of this construction. 
	
	By our choice of $K'$, there is a homeomorphism $\phi$ on $(S,K')$ representing $\bar{\gamma}$ such that $\phi$ genuinely preserves each $D_i$ and fixes its boundary pointwise.
	Thus the restriction $\phi|_{D_i}$ represents a mapping class $\gamma_i$ in $\PMod(D_i,K'_i)$.
	Applying Lemma~\ref{lemma:nested_Dehn_twists} to $\gamma_i\in \PMod(D_i,K'_i)$ gives rise to an infinite sequence 
	of points $Z(i)=(z_0(i),z_1(i),\cdots)$ in $K_i$ and a nesting sequence of dividing disks $D_i=E_0(i)\supset E_1(i)\supset\cdots$,
	such that $\phi|_{D_i}$ as a mapping class in $\PMod(D_i,Z(i))$ is represented by an infinite product $\prod_{j\ge0} t_j(i)^{m_j(i)}$, where $t_j(i)$ is a (clockwise) Dehn twist around $\partial E_j(i)$.
	
	Let $X_i=\{z_0(i),\cdots, z_M(i)\}$. Then the image of $\phi|_{D_i}$ in $\PMod(D_i,X_i)$ is $\prod_{j=0}^{M-1} t_j(i)^{m_j(i)}$.
	This mapping class is encoded by the sequence of integers $(m_0(i),m_1(i),\cdots,m_M(i))$ that
	records the power of the Dehn twists around the boundaries of the nested dividing disks. If we could somehow arrange
	for $m_1(i)=m_2(i)=\cdots = m_{M-1}(i)=0$ then the image of $\phi|_{D_i}$ in $\Mod(\intrm(D_i),X_i)$ would be
	the identity, as desired in bullet (\ref{item: restrict to Dehn twists}).
	Our goal is to modify the sequence $(m_0(i), m_1(i),\cdots,)$ to arrive at this case by passing to subsequences of $Z(i)$ and taking (products of) conjugates of $\phi|_{D_i}$ by mapping classes $g_i$ in $\Mod(D_i,K'_i)$.
	Since the modifications can be done simultaneously for different $i$'s without interfering with each other, we fix some $i$ below and omit (in our notation) 
	the dependence of $Z(i)$, $z_j(i)$, $m_j(i)$, $t_j(i)$ and $E_j(i)$ on $i$ for simplicity.
	
	The most important operation is to replace the sequence of points $Z$ by an infinite subsequence $Z'\subset Z$. This has a
	predictable effect on the associated integer sequence that we now describe.
	If we take an infinite subsequence $Z'=(z_{j_0},z_{j_1},z_{j_2}\cdots)$ of $Z$, 
	then the image of $\phi|_{D_i}$ in $\PMod(D_i,Z')$ is also isotopic to an infinite product of Dehn twists $\prod_{k=0}^{\infty} t_{j_k}^{m'_k}$,
	where $m'_k=\sum_{j_{k-1}< j\le j_k} m_j$ and $j_{-1}\defeq -1$. We refer to such an operation on associated integer
	sequences $(m_0,m_1,m_2,\cdots)\mapsto (m'_0,m'_1,m'_2,\cdots)$ as an {\em amalgamation}.
	In this situation, there is a homeomorphism $g_i$ supported on $D_i$ preserving the Cantor set $K'_i$ such that $g_i(z_{j_k})=z_k$ and $g_i(E_{j_k})=E_k$ for all $k\ge1$, $g_i(z_0)=z_0$,
	and $g_i(E_{j_0})$ is isotopic to $E_0=D_i$ in $(D_i,Z)$ (i.e. $g_i(E_{j_0})$ contains all points in $Z$). 
	Then the conjugate $g_i \phi g_i^{-1}$ represents a mapping class $\phi(Z')$ in $\PMod(D_i,Z)$ and is isotopic to $\prod_{k=0}^{\infty} t_k^{m'_k}$.
	
	Consider the specific subsequences $Z'_0=(z_1,z_2,z_3,\cdots)$ and $Z'_1=(z_0,z_2,z_3,z_4,\cdots,)$ obtained by omitting $z_0$ and $z_1$ respectively.
	Then the procedure above gives two corresponding conjugates $g_{i,0}\phi g_{i,1}^{-1}$ and $g_{i,1}\phi g_{i,1}^{-1}$ of $\phi$ (preserving the family $D_Q$),
	whose images in $\PMod(D_i,Z)$ are mapping classes $\phi(Z'_0)$ and $\phi(Z'_1)$ whose associated sequences (that appear as the exponents of $t_k$'s) are
	$(m_0+m_1,m_2,m_3,\cdots)$ and $(m_0,m_1+m_2,m_3,\cdots)$ respectively.
	Therefore, the homeomorphism $(g_{i,0}\phi g_{i,0}^{-1}) \cdot (g_{i,1}\phi g_{i,1}^{-1})^{-1}$ has the following properties:
	\begin{enumerate}
		\item it is the identity in $\PMod(S,Q)$ since the restrictions of both $g_{i,0}\phi g_{i,0}^{-1}$ and $g_{i,1}\phi g_{i,1}^{-1}$ to $Q$ are equal to $\bar{\gamma}|_Q$,
		\item it preserves all disks $D_j$ in $D_Q$,
		\item it is isotopic to $t_0^{m_1}\cdot t_1^{-m_1}$ as a mapping class in $\Mod(D_i,Z(i))$, and is the identity in all $D_j$ for $j\neq i$, and
		\item it represents an element in $\overline{N}$ as a mapping class in $\PMod(S,K')\le \Mod(S,K')$ since $\overline{N}$ is $\Mod(S,K')$-normal.
	\end{enumerate}
	
	Multiplying this homeomorphism with $\phi$, we obtain a mapping class in $\overline{N}$ with the following properties:
	\begin{enumerate}
		\item its image in $\PMod(S,Q)$ is $\bar{\gamma}|_Q$,
		\item it preserves all disks in $D_Q$, and
		\item it is isotopic to the infinite product $t_0^{m_0+m_1} \cdot \prod_{k\ge2} t_k^{m_k}$ in $(D_i,Z(i))$ and is the same as $\phi$ in any $(D_j,Z(j))$ for $j\neq i$.
	\end{enumerate}
	Thus the net effect of this entire procedure is to change $m_1$ to $0$ and $m_0$ to $m_0 + m_1$; i.e.\/ its effect is
	represented on sequences as $(m_0,m_1,m_2,\cdots)\mapsto (m_0+m_1,0,m_2,\cdots)$.
	
	Replacing the role of the indices $0$ and $1$ above by some $j$ and $j+1$, we can similarly 
	change $m_{j+1}$ to $0$ and $m_j$ to $m_j+m_{j+1}$. By applying such operations inductively, 
	we can modify the sequence until it is of the form 
	$$(\sum_{k=0}^{M-1} m_k, 0,0,\cdots, 0, m_M, m_{M+1},\cdots)$$ i.e.\/ all integers are zeros 
	at the $j$-th location for all $1\le j\le M-1$.
	As discussed above, the mapping class associated to this sequence 
	restricts to the identity in each $\Mod(\intrm(D_i),X_i)$ where $X_i=\{z_0(i),\cdots, z_M(i)\}$.
	
	Performing this modification simultaneously for all $D_i$'s we obtain an element 
	$\bar{\gamma}'$ in $\overline{N}$ with the desired properties.	
\end{proof}

This completes the proof of Proposition~\ref{proposition:local_behaviour} and therefore also 
Lemma~\ref{lemma:restriction_is_inert}. Together with Lemma~\ref{lemma:inert_is_realized} this
completes the proof of the Inertia Theorem~\ref{theorem:inert}.

\begin{corr}
Every inert subgroup of $\PMod(S,n)$ is algebraically inert.
\end{corr}
\begin{proof}
Every inert subgroup is equal to $N_n$ for some normal $N$. But then for any $m>n$, if
$N_m \subset \PMod(S,m)$ is in the spectrum of $N$, then the image of $N_m$ in $\PMod(S,n)$ is
equal to $N_n$ for every forgetful homomorphism from $\PMod(S,m) \to \PMod(S,n)$; thus
$N_n$ is algebraically inert.
\end{proof}

It seems unlikely that the converse is true --- i.e.\/ that every algebraically inert subgroup is inert ---
but we do not know a counterexample.

\subsection{Operations on normal subgroups}

We give two examples of operations on normal subgroups that preserve the spectrum.

\subsubsection{Inflation}

The self-similarity of a Cantor set gives rise to a natural operation on pure normal
subgroups of $\Mod(S,K)$ that we call {\em inflation}.

\begin{defn}
Let $K \subset S$ be a Cantor set and let $K' \subset K$ be a proper Cantor subset.
For $N$ a pure normal subgroup of $\Mod(S,K)$, the {\em inflation} of $N$, denoted
$N^+$, is the pure normal subgroup of $\Mod(S,K)$ obtained by choosing a homeomorphism
of pairs $h:(S,K') \to (S,K)$ and an associated isomorphism $h_*:\Mod(S,K') \to \Mod(S,K)$
and then defining $N^+$ to be the image of $N$ under the composition of
the restriction $\PMod(S,K) \to \PMod(S,K')$ with $h_*$.
\end{defn}

Because $N$ is normal in $\Mod(S,K)$ it follows that the definition of $N^+$ does not depend 
on the choice of $K' \subset K$ or the choice of homeomorphism $h$.

One nice property of inflation is that it preserves spectrum:

\begin{prop}\label{prop:inflation_spectrum}
For any pure normal $N$ the groups $N$ and $N^+$ have the same spectrum.
\end{prop}
\begin{proof}
For any $n$ and any $n$-element subset $Q\subset K$ there is some homeomorphism of $S$
taking $K$ to $K$ and $Q$ into $K'$. The proposition follows.
\end{proof}

\begin{exmp}
Let $N$ be the normal subgroup generated by $p$-th powers of Dehn twists in the boundary
of a dividing disk. Then $N^+$ is the normal subgroup generated by $p$-th powers of Dehn 
twists in all embedded loops of $S\setminus K$ that are homotopically trivial in $S$.
It differs from $N$ by including in the generating set twists around those homotopically trivial loops that enclose all of $K$. 
For sufficiently large $p$ these groups are distinct as can be seen e.g.\/
by the method of Funar~\cite{Funar}.
\end{exmp}

In the examples of normal subgroups $N$ we have encountered, $N$ is a subgroup of its
inflation $N^+$. We believe there could be examples where $N^+$ does not contain $N$
but we do not know of any.

\subsubsection{Closure}

The group $\Mod(S,K)$ has a natural topology in which a sequence $\gamma_i \in \Mod(S,K)$
converges to the identity if there are representative homeomorphisms $h_i$ of $(S,K)$ 
that converge to the identity in the compact--open topology. This happens for example if each $\gamma_i$
has a representative supported in some compact subsurface $R_i \subset S$ where $R_{i+1} \subset R_i$ and
$\cap_i R_i$ is totally disconnected. 

If $N$ is a pure normal subgroup then so is its closure $\bar{N}$. Furthermore, if
$x \subset K$ is any $n$-element set, then for any convergent sequence 
$\gamma_i \to \gamma$ in $\bar{N}$ the restrictions of $\gamma_i$ to $\PMod(S,x)$ are
eventually constant. Thus:

\begin{prop}\label{prop:closure_spectrum}
For any pure normal $N$ the groups $N$ and $\bar{N}$ have the same spectrum.
\end{prop}

It seems very plausible that there could be infinitely many (and maybe even uncountably many)
normal subgroups with the same spectrum. One would like to develop new tools to construct
and distinguish them.

\end{document}